\documentclass[12pt]{article}
\usepackage{amsfonts, amssymb, amsthm}

\newtheorem{theorem}{Theorem}

\newtheorem{corollary}{Corollary}
\theoremstyle{definition}

\newcommand{\PP}{\mathcal{P}}
\newcommand{\UU}{\mathcal{U}}
\newcommand{\V}{\mathcal{V}}

\begin{document}

\title{Selective and Ramsey ultrafilters on $G$-spaces }
\author{O.V. Petrenko, I.V. Protasov}
\date{}                                                                                                  

\maketitle

\begin{abstract}
Let $G$ be a group, $X$ be an infinite transitive $G$-space. A free ultrafilter $\UU$ on $X$ is called $G$-selective if, for any $G$-invariant partition $\PP$ of $X$, either one cell of $\PP$ is a member of $\UU$, or there is a member of $\UU$ which meets each cell of $\PP$ in at most one point. We show (Theorem 1) that in ZFC with no additional set-theoretical assumptions there exists a $G$-selective ultrafilter on $X$, describe all $G$-spaces $X$ (Theorem 2) such that each free ultrafilter on $X$ is $G$-selective, and prove (Theorem 3) that a free ultrafilter $\UU$ on $\omega$ is selective if and only if $\UU$ is $G$-selective with respect to the action of any countable group $G$ of permutations of $\omega$.

A free ultrafilter $\UU$ on $X$ is called $G$-Ramsey if, for any $G$-invariant coloring $\chi:[G]^2 \to \{0,1\}$, there is $U\in \UU$ such that $[U]^2$ is $\chi$-monochrome. By Theorem 4, each $G$-Ramsey ultrafilter on $X$ is $G$-selective. Theorems 5 and 6 give us a plenty of $\mathbb{Z}$-selective ultrafilters on $\mathbb{Z}$ (as a regular $\mathbb{Z}$-space) but not $\mathbb{Z}$-Ramsey. We conjecture that each $\mathbb{Z}$-Ramsey ultrafilter is selective.

\

{\bf 2010 AMS Classification}: 05D10, 54H15

\


\end{abstract}

A free ultrafilter $\UU$ on an infinite set $X$ is said to be {\it selective} if, for any partition $\PP$ of $X$, either one cell of $\PP$ is a member of $\UU$, or some member of $\UU$ meets each cell of $\PP$ in at most one point. The selective ultrafilters on $\omega = \{0,1,\ldots\}$ are also known under the name {\it Ramsey ultrafilters} (see, for example \cite{b1}) because $\UU$ is selective if and only if, for each coloring $\chi:[\omega]^2 \to \{0,1\}$ of $2$-element subsets of $\omega$, there exists $U\in \UU$ such that the restriction $\chi|_{[U]^2}\equiv const$.

Let $G$ be a group, $X$ be a $G$-space with the action $G\times X\to X$, $(g,x) \mapsto gx$. All $G$-spaces under consideration are supposed to be transitive: for any $x,y \in X$, there exists $g\in G$ such that $gx = y$. If $G=X$ and $gx$ is the product of $g$ and $x$ in $G$, $X$ is called a {\it regular $G$-space}. A partition $\PP$ of a $G$-space $X$ is {\it $G$-invariant} if $gP\in \PP$ for all $g\in G$, $P\in \PP$.

Now let $X$ be an infinite $G$-space. We say that a free ultrafilter $\UU$ on $X$ is {\it $G$-selective} if, for any $G$-invariant partition $\PP$ of $X$, either some cell of $\PP$ is a member of $\UU$, or there exists $U\in \UU$ such that $|P\cap U|\leqslant 1$ for each $P\in\PP$. Clearly, each selective ultrafilter on $X$ is $G$-selective.

The selective ultrafilters on $\omega$ exist under some additional to ZFC set-theoretical assumptions (say, the continuum hypothesis CH), but there are models of ZFC with no selective ultrafilters (see \cite{b1}). In contrast to these facts, we show (Theorem 1) that a $G$-selective ultrafilter exists on any infinite $G$-space $X$. Then we characterize (Theorem 2) all $G$-spaces $X$ such that each free ultrafilter on $X$ is $G$-selective, and show (Theorem 3) that a free ultrafilter $\UU$ on $\omega$ is $G$-selective for any transitive group $G$ of permutations on $\omega$ if and only if $\UU$ is selective.

For a $G$-space $X$ and $n\geqslant 2$, a coloring  $\chi:[X]^n \to \{0,1\}$ is said to be {\it $G$-invariant} if, for any $\{x_1, \ldots, x_n\} \in [X]^n$ and $g\in G$, $\chi(\{x_1, \ldots, x_n\})=\chi(\{gx_1, \ldots, gx_n\})$. We say that a free ultrafilter $\UU$ on $X$ is {\it $(G,n)$-Ramsey} if, for every $G$-invariant coloring $\chi:[X]^n \to \{0,1\}$, there exists $U\in \UU$ such that $\chi|_{[U]^n}\equiv const$. In the case $n=2$, we write "$G$-Ramsey" instead of "$(G,2)$-Ramsey".

We show (Theorem 4) that every $G$-Ramsey ultrafilter is $G$-selective, but the converse statement is very far from truth. Theorems 5 and 6 give us a plenty ultrafilters on $\mathbb{Z}$ (as a regular $\mathbb{Z}$-space, $\mathbb{Z}$ is the group of integers) which are not $\mathbb{Z}$-Ramsey, while each free ultrafilter on $\mathbb{Z}$ is $\mathbb{Z}$-selective. Moreover, we conjecture, that each $\mathbb{Z}$-Ramsey ultrafilter on $\mathbb{Z}$ is selective. By Corollary 5, each $(\mathbb{Z},4)$-Ramsey ultrafilter is selective.

A $B$-Ramsey ultrafilter on the countable Boolean group $B = \oplus_{\omega}\mathbb{Z}_2$ needs not to be selective, but a $B$-Ramsey ultrafilter cannot be constructed in ZFC without additional assumptions.

\section{Selective ultrafilters}

Let $X$ be a $G$-space, $x_0\in X$. We put $St(x_0) = \{y\in G: gx_0 = x_0\}$ and identify $X$ with the left coset space $G/St(x_0)$ of $G$ by $St(x_0)$. If $\PP$ is a $G$-invariant partition of $X = G/S, S = St(x_0)$, we take $P_0\in \PP$ such that $S\in P_0$, put $H = \{g\in G: gS \in P_0\}$ and note that the subgroup $H$ completely determines $\PP$: $xS, yS \in G/S$ are in the same cell of $\PP$ if and only if $xy^{-1}\in H$. Thus, $\PP = \{x(H/S): x\in L\}$ where $L$ is a set of representatives of the left cosets of $G$ by $H$.

\begin{theorem}
For every infinite $G$-space $X$, there exists a $G$-selective ultrafilter $\UU$ on $X$.
\end{theorem}

\begin{proof}
Let $S$ be a subgroup of $G$, $X = G/S$.We consider two cases.

Case 1. There exists a subgroup $T$ of $G$ such that $S\subset T$, $|T:S|=\infty$ and $|F:S|< \infty$ for each subgroup $F$, $S\subseteq F \subset T$. We take an arbitrary free ultrafilter $\UU$ on $G/S$ such that $T/S\subseteq \UU$. To show that $\UU$ is $G$-selective, we take an arbitrary $G$-invariant partition $\PP$ of $G/S$ and choose a subgroup $H$, $S\subseteq H$ which determines $\PP$. If $T\subseteq H$ then $H/S\in \UU$. Otherwise, we put $F = T\cap S$ and note that $|F:S|< \infty$. Decompose $T/S$ into left cosets by $F/S$, note that each coset has $|F/S|$ elements and choose $U\in \UU$ such that $H/S \subseteq \UU$ and $U$ meets each coset in at most one element. Then $|U\cap P|\leqslant 1$ for each cell $P$ of the partition $\PP$.

Case 2. For each subgroup $T$, $S\subset T$, there exists a subgroup $T'$ such that $S\subset T' \subset T$ and $|T':S|=\infty$. We choose maximal linearly ordered by $\subseteq$ family $\mathcal{F}$ of subgroups of $G$ such that, for each $T\in \mathcal{F}$, $S\subset T$ and $|T:S|=\infty$. We put $F = \cap\mathcal{F}$ and note that $|F:H|<\infty$. Then we take an arbitrary ultrafilter $\UU$ on $G/S$ such that $\{T/S: T\in \mathcal{F}\}\subseteq \UU$. Let $P$ be a $G$-invariant partition of $G/S$ determined by some subgroup $H$, $S\subseteq H$. If $H\in \mathcal{F}$ then $H/S\in \UU$. Otherwise, by maximality of $\mathcal{F}$, there exists $T\in \mathcal{F}$ such that $T\cap H \subseteq F$. Then we choose $U\in \UU$ such that $T/S\in \UU$ and $U$ meets each left coset of $T/S$ by $F/S$ in at most one point. Clearly, $|P\cap U|\leqslant 1$ for each cell of $\PP$.
\end{proof}

\begin{theorem}
Let $G$ be a group, $S$ be a subgroup of $G$ such that $|G:S|=\infty$. $X = G/S$. Each free ultrafilter on $X$ is $G$-selective if and only if, for each subgroup $T$ of $G$ such that $S\subset T \subset G$, either $|T:S|$ is finite or $|G:T|$ is finite.
\end{theorem}

\begin{proof}
Suppose that there exists a subgroup $T$ of $G$ such that $S\subset T \subset G$ and $|T:S|=\infty$, $|G:T|=\infty$. We pick a family $\{g_nT: n\in \omega\}$ of distinct cosets of $G$ by $T$ and, using the Zorn's lemma, choose a maximal family $\UU$ of subsets of $G/S$ such that, for each $U\in \UU$,
$$\{n\in \omega: U\cap x_n(T/S)\mbox{ is infinite} \}$$
is infinite. Clearly, $\UU$ is an ultrafilter and, by the construction, each $U\in \UU$ meets infinitely many members of the $G$-invariant partition $\PP$ determined by $T$ in infinitely many points, so $\UU$ is not $G$-selective.

On the other hand, if $|T:S|<\infty$ then the $G$-invariant partition $\PP$ determined by $T$ consists of finite sets of cardinality $|T:S|$. If $|G:T|<\infty$ then $\PP$ is a finite partition. Therefore, each free ultrafilter of $G/S$ is $G$-selective.
\end{proof}

Let $G$ be an infinite Abelian group such that, for each subgroup $S$ of $G$ either $S$ is finite or $G/S$ is finite. If $G$ has an element of infinite order then $G$ is isomorphic to $\mathbb{Z}\times F$ where $F$ is finite. If $G$ is a torsion group then $G$ is isomorphic to $\mathbb{Z}_{p^\infty}\times F$ where $\mathbb{Z}_{p^\infty}$ is the Pr\"uffer $p$-group (see \cite[\S 3]{b3}), $F$ is finite. This is an elementary exercise on Abelian groups. Thus, the class of Abelian groups $G$ such that each ultrafilter on $G$ is $G$-selective is very narrow.

\begin{theorem}
If a free ultrafilter $\UU$ on $\omega$ is $G$-selective with respect to the action of any transitive group $G$ of permutations of $\omega$ then $\UU$ is selective.
\end{theorem}

\begin{proof}
Let $\PP$ be a partition of $\omega$ such that each member of $\PP$ is not a member of $\UU$. We state that $\PP$ can be partitioned $\PP = \cup_{n\in \omega}\PP_n$ so that, for each $n\in \omega$, $\cup \PP_n$ is infinite and is not a member of $\UU$. If the set $\PP'$ of all finite blocks of $\PP$ is finite, we take an arbitrary infinite block $P_0$, put $\PP_0 = \{\PP', \{P_0\}\}$ and enumerate $\PP_1, \PP_2, \ldots$ all remaining infinite blocks of $\PP$. If $\PP'$ is infinite, we partition $\PP' = \PP'_0 \cup \PP'_1$ such that $\PP'_0$ and $\PP'_1$ are infinite. We take $i \in \{0,1\}$ (say $i=0$) such that $\cup \PP'_0 \notin \UU$. Then we repeat this procedure for $\PP'_1$ and so on. After $\omega$ steps, we get a desired partition of $\PP'$. After enumeration of $\{\PP'_n: n\in\omega\}$ of infinite blocks of $\PP$, we obtain a desired partition of $\PP$.

For each $n\in \omega$, we put $Q_n = \cup \PP_n$, take an arbitrary countable group $G = \{g_n: n\in \omega\}$ and identify $\omega$ with $G\times G$ so that $Q_n = \{g_n\}\times G$, $n\in \omega$. We consider $G\times G$ as a regular $(G\times G)$-space and note that the partition $\{Q_n: n\in \omega\}$ of $G\times G$ is ($G\times G$)-invariant. Since $\UU$ is ($G\times G$)-selective, there exists $U\in \UU$ such that $|U\cap Q_n|\leqslant 1$ for each $n\in \omega$. By the construction of $Q_n$, $|U\cap P|\leqslant 1$ for each $P\in \PP$. Hence, $\UU$ is selective.
\end{proof}

\section{Ramsey ultrafilters}

\begin{theorem}
For a $G$-space $X$, each $G$-Ramsey ultrafilter on $X$ is $G$-selective.
\end{theorem}

\begin{proof}
Let $\PP$ be a $G$-invariant partition of $X$. We define a coloring $\chi:[X]^2 \to \{0,1\}$ by the rule: $\chi(\{x,y\})=0$ if and only if $x,y$ are in the same cell of the partition $\PP$. Since $\PP$ is $G$-invariant, $\chi$ is also $G$-invariant. We take $U\in \UU$ such that $\chi|_{[U]^2}\equiv i$ for some $i\in \{0,1\}$. If $i=0$ and $x\in U$ then $U$ is contained in the block $P$ of $\PP$ such that $x\in P$. If $i=1$ then $U$ meets each block of $\PP$ in at most one point. Hence, $\UU$ is $G$-selective.
\end{proof}

Let $G$ be a group with the identity $e$. Each $G$-invariant 2-coloring of the regular $G$-space can be described as follows. We say that a coloring $\chi':G\setminus\{e\} \to \{0,1\}$ is {\it symmetric} if $\chi'(x) = \chi'(x^{-1})$ for each $x\in G \setminus \{e\}$. Then we put $\chi(\{x,y\}) = \chi'(x^{-1}y)$ and note that $\chi(\{gx,gy\})=\chi(\{x,y\})$ for all $\{x,y\}\in [G]^2$ and $g\in G$. On the other hand, if a coloring $\chi: [G]^2 \to \{0,1\}$ is $G$-invariant then the coloring $\chi': G\setminus\{e\}\to \{0,1\}$, $\chi'(x) = \chi(\{e,x\})$ is symmetric and uniquely determines $\chi$.

We fix an arbitrary linear ordering $\leqslant$ of $G$ and, for each subset $U$ of $G$, put $D(U) = \{x^{-1}y: x,y\in U, x < y \}$. For an ultrafilter $\UU$ on $G$, we define a family $D(\UU)$ of subsets of $G$ by

$$V\in D(\UU)\Leftrightarrow  \exists U \in \UU: D(U)\subseteq V.$$

We use also the product $\V\UU$ of ultrafilters on $G$ defined as follows (see \cite[Chapter 4]{b4}). We take an arbitrary $V\in \UU$ and, for each $g\in V$, pick $U_g\in \UU$. Then $\cup_{g\in V}gU_g$ is a member of $\V\UU$, and each member of the ultrafilter $\V\UU$ contains a subset of this form. We denote $\UU^{-1} = \{ U^{-1} : U\in \UU\}$. 

\begin{theorem}
Let $\leqslant$ be the natural linear ordering of $\mathbb{Z}$, $\mathbb{Z}^+ = \{z\in \mathbb{Z}: z > 0\}$, $\UU$ be a free ultrafilter on $\mathbb{Z}$ such that $\mathbb{Z}^+ \in \UU$. Then the following statements hold:

\begin{itemize}
\item[(i)] $D(\UU) \subseteq (-\UU)+\UU$;
\item [(ii)] $\UU$ is $\mathbb{Z}$-Ramsey if and only if $D(\UU) = (-\UU)+\UU$.

\end{itemize}
\end{theorem}

\begin{proof}
$(i)$ We take an arbitrary $U\in \UU$ such that $U\subseteq \mathbb{Z}^+ $. For each $z\in U$, put $U(z) = \{x\in U: x > z\}$. Then $D(U)= \cup_{z\in U}(-z + U(z))$. Since $U(z) \in \UU$, by the definitions of $-\UU$ and $(-\UU) +\UU$, we have $D(\UU)\subseteq (-\UU) +\UU$.

$(ii)$ Assume that $\UU$ is $\mathbb{Z}$-Ramsey and take $U\in \UU$, $U \subseteq \mathbb{Z}^+$. For each $z\in U$, we pick an arbitrary $U_z \in \UU$ such that $z<x$ for each $x\in U$. Then we put $W = \cup_{z\in U} (-z + U_z)$ and define a symmetric coloring $\chi': \mathbb{Z}\setminus \{0\} \to \{0,1\}$. If $x\in W \cup (-W)$, we put $\chi'(x) = 0$, otherwise, $\chi'(x) =1$. We take a coloring $\chi : [\mathbb{Z}]^2 \to \{0,1\}$ determined by $\chi'$. Since $\UU$ is $\mathbb{Z}$-Ramsey, there is $V\in \UU$, $V\subseteq U$ such that $\chi|_{[V]^2}\equiv i$ for some $i\in \{0,1\}$. By the definition of $\chi'$, $x=0$ and $D(V)\subseteq W$. Hence, $W\in D(\UU)$ so $(-\UU) +\UU \subseteq D(\UU)$. By $(i)$, $D(\UU) \subseteq (-\UU)+\UU$ so $D(\UU) = (-\UU)+\UU$.

On the other hand, let $D(\UU) = (-\UU)+\UU$. We consider an arbitrary symmetric coloring $\chi' :\mathbb{Z}\setminus \{0\} \to \{0,1\} $ and denote by $\chi$ corresponding coloring of $[\mathbb{Z}]^2$. Since $(-\UU) + \UU$ is an ultrafilter, there is $W \in (-\UU) + \UU$, $W\subseteq \mathbb{Z}^+$ such that $\chi'|_W \equiv i$, $i\in \{0,1\}$. We take $V\in \UU$ such that $D(V) \subseteq W$. Then $\chi|_{[V]^2}\equiv i$ so $\UU$ is $\mathbb{Z}$-Ramsey.
\end{proof}

Let $G$ be a discrete group. The Stone-\v{C}ech compactification $\beta G$ of $G$ can be identified with the set of all ultrafilters on $G$ and $\beta G$ with above defined multiplication is a semigroup which has the minimal ideal $K(\beta G)$ (see \cite[Chapter 6]{b4}).

\begin{corollary}
Each ultrafilter $\UU$ from the closure $cl\; K(\beta \mathbb{Z})$ is not $\mathbb{Z}$-Ramsey.
\end{corollary}

\begin{proof}
On the contrary, suppose that some ultrafilter $\UU \in cl \; K(\beta \mathbb{Z})$ is $\mathbb{Z}$-Ramsey. Since $\UU \in cl\; K(\beta \mathbb{Z})$, by \cite[Corollary 5.0.28]{b2}, for every $U\in \UU$, there exists a finite subset $K$ such that $\mathbb{Z} = K + U - U$. We note that $U-U = D(U) \cup (-D(U)) \cup\{0\}$. Now we partition $\mathbb{Z}^+ = Z_0 \cup Z_1$, 

$$Z_0 = \bigcup_{n\in \omega}[2^{2n}, 2^{2n+1}), \quad Z_1 = \mathbb{Z}^+ \setminus Z_0,$$
and applying Theorem 5 $(ii)$, choose $U\in \UU$ and $i\in \{0,1\}$ such that $D(U)\subseteq Z_i$. Clearly, $F + U - U \ne \mathbb{Z}$ for each finite subset $F$ of $\mathbb{Z}$. Hence, $\UU \notin K(\beta \mathbb{Z})$ and we get a contradiction.
\end{proof}

A free ultrafilter $\UU$ on an Abelian group $G$ is said to be a {\it Schur ultrafilter} if, for any $U\in \UU$, there are distinct $x,y\in U$ such that $x+y \in U$. We note that each idempotent from $\beta \mathbb{Z} \setminus \mathbb{Z}$ is a Schur ultrafilter.

\begin{corollary}
Each Schur ultrafilter $\UU$ on $\mathbb{Z}$ is not $\mathbb{Z}$-Ramsey.
\end{corollary}

\begin{proof}
On the contrary, we suppose that $\UU$ is $\mathbb{Z}$-Ramsey and $\mathbb{Z}^+ \in \UU$. Since $\UU$ is a Schur ultrafilter, by Theorem 5, $D(\UU) = \UU = -\UU + \UU$. By \cite[Corollary 13.19]{b4}, $(-\UU) + \UU \ne \UU $ for every free ultrafilter $\UU$ on $\mathbb{Z}$.
\end{proof}

A free ultrafilter $\UU$ on $\mathbb{Z}$ is called {\it prime} if $\UU$ cannot be represented as a sum of two free ultrafilters.

\begin{corollary}
Every $\mathbb{Z}$-Ramsey ultrafilter on $\mathbb{Z}$ is prime.
\end{corollary}

\begin{proof}
We need two auxiliary claims.

Claim 1. If $\UU,\V$ are free ultrafilters and $\UU + \V$ is $\mathbb{Z}$-Ramsey then $D(\UU+\V) = D(\UU) = D(\V)$, in particular (see Theorem 5) $\UU$ and $\V$ are $\mathbb{Z}$-Ramsey.

Let $\mathcal{W} = \UU + \V$, $U\in \UU$, $V_x\in \V$, $x\in U$ and $W = \bigcup_{x\in U}xV_x$. To see that $D(\V) = D(\mathcal{W})$, we fix $x\in U$ and put $V'_x =\{y\in V: y > x\}$. If $y_1, y_2 \in V_x$ and $y_2 > y_1$ then $y_2 - y_1 = (x+y_2) - (x + y_1)$ so $D(V_x)\subseteq D(W)$ and $D(\mathcal{W}) = D(\V)$ because $D(\mathcal{W})$ is an ultrafilter.

To show that $D(\UU) = D(\mathcal{W})$, we take $x_1, x_2 \in U$, $x_1 < x_2$ and pick an arbitrary $y\in V_{x_1} \cap V_{x_2}$. Since $x_2 - x_1 = (x_2+y) - (x_1 + y)$ and $x_1 + y, x_2 + y \in W$, $D(U)\subseteq D(W)$ so $D(\mathcal{W}) = D(\UU)$.

Claim 2. If $\mathcal{W}$ is $\mathbb{Z}$-Ramsey then $\mathcal{W}$ is a right cancellable element of the semigroup $\beta \mathbb{Z}$.

If not, by \cite[Theorem 8.18]{b4}, $\mathcal{W} = \UU + \mathcal{W}$ for some idempotent $\UU$. By Claim 1, $\UU$ is $\mathbb{Z}$-Ramsey which contradicts Corollary 2.

At last, suppose that some $\mathbb{Z}$-Ramsey ultrafilter $\mathcal{W}$ is represented as $\mathcal{W} = \UU + \V$. Applying Theorem 5 and Claim 1, we get $D(\mathcal{W}) = D(\UU) = D(\V)$ and 

$$D(\mathcal{W}) = (-\UU) + (-\V) + \UU + \V, \quad D(\V) = (-\V) + \V, \quad D(\UU) = (-\UU) + \UU.$$
By Claim 2, $(-\UU) + (-\V) + \UU = (-\V)$. It follows that $\mathbb{Z}^+ \in \UU$ if and only if $\mathbb{Z}^+ \notin \V$. On the other hand $(-\UU) + \UU = (-\V) + \V$. So, $\mathbb{Z}^+ \in \UU$ if and only if $\mathbb{Z}^+ \in \V$. Hence, $\mathcal{W}$ is prime.
\end{proof}

We do not know whether every $\mathbb{Z}$-Ramsey ultrafilter $\UU$ is strongly prime, i.e. $\UU$ does not lie in the closure of the set $\mathbb{Z}^* + \mathbb{Z}^*$. A free ultrafilter $\UU$ on a group $G$ is strongly prime if and only if some member of $\UU$ is sparse. A subset $S$ of an infinite group $G$ is called {\it sparse} \cite{b5} if, for every infinite subset $X$ of $G$, there exists a finite subset $F\subset X$ such that $\cap_{g\in F} gS$ is finite.

Following \cite{b6}, we say that a subset $A$ of a group $G$ is {\it $k$-thin}, $k\in \mathbb{N}$ if 
$$|gA\cap A|\leqslant k$$
for each $g\in G\setminus \{e\}$.

\begin{theorem}
Let $\UU$ be a $\mathbb{Z}$-Ramsey ultrafilter on $\mathbb{Z}$, $\mathbb{Z}^+ \in \UU$. If there exists a 1-thin subset $A$ of $G$ such that $A\in \UU$ then $\UU$ is selective.
\end{theorem}

\begin{proof}
We fix an arbitrary coloring $\varphi: [\mathbb{Z}]^2 \to \{0,1\}$ and define a symmetric coloring $\chi' : \mathbb{Z}\setminus \{0\}$ as follows. If $g\in \mathbb{Z}\setminus \{0\}$ and there are $a,b\in A$, $a<b$ such that $g=b-a$, we put $\chi'(g) = \chi'(-g) = \varphi(\{a,b\})$. Otherwise, $\chi'(g) = \chi'(-g) = 1$. This definition is correct because $A$ is 1-thin. Then we consider a coloring $\chi: [\mathbb{Z}]^2 \to \{0,1\}$ determined by $\chi'$. Since $\UU$ is $\mathbb{Z}$-Ramsey, there exists $U\in \UU$, $U\subseteq A$ such that $\chi|_{[U]^2} \equiv const$. By the construction of $\chi$, we have $\chi|_{[U]^2} \equiv \varphi|_{[U]^2}$. Thus, $\varphi|_{[U]^2} \equiv const$ and $\UU$ is selective.
\end{proof}

Recall that a free ultrafilter $\UU$ on $\mathbb{Z}$ is a {\it $Q$-point} if, for every partition $\PP$ of $\mathbb{Z}$ into finite cells there is a member of $\PP$ which meets each cell in at most one point.

\begin{corollary}
If a free ultrafilter $\UU$ on $\mathbb{Z}$ is $\mathbb{Z}$-Ramsey and a $Q$-point then $\UU$ is selective.
\end{corollary}

\begin{proof}
To apply Theorem 6, it suffices to show that $\UU$ has a 1-thin member. We suppose that $\mathbb{Z}^+ \in \UU$, use the partition $\mathbb{Z}^+ = Z_0 \cup Z_1$ from Corollary 1, and take $i \in \{1,2\}$ and $U \in \UU$ such that $U$ meets each cell $[2^m, 2^{m+1})$ of $Z_i$ in at most one point. Clearly, $U$ is 1-thin.
\end{proof}

We do not know if each $P$-point in $\mathbb{Z}^*$ is $\mathbb{Z}$-Ramsey. Recall that $\UU$ is a $P$-point if, for every partition $\PP$ of $\mathbb{Z}$, either some cell of $\PP$ is a member of $\UU$, or there exists $U\in \UU$ such that $U\cap P$ is finite for each $P\in \PP$.

In the proof of the next corollary, we use the following observation: if $\UU$ is $(\mathbb{Z}, n)$-Ramsey and $m < n$, then $\UU$ is $(\mathbb{Z}, m)$-Ramsey.

\begin{corollary}
Each $(\mathbb{Z}, 4)$-Ramsey ultrafilter $\UU$ on $\mathbb{Z}$ is selective.
\end{corollary}

\begin{proof}
Since $\UU$ is $(\mathbb{Z}, 2)$-Ramsey, to apply Theorem 6, it suffices to find a 1-thin member of $\UU$.

We define a coloring $\chi_1 : [\mathbb{Z}]^4 \to \{0,1\}$ by the rule: $\chi_1(F) = 0$ if and only if there is a numeration $F = \{x,y,z,t\}$ such that $x+y = z+t$. Since $\chi_1$ is $\mathbb{Z}$-invariant, there is $Y\in \UU$ such that $\chi_1|_{[Y]^4}\equiv i$. Since $A$ is infinite, $i=1$. 

Then we define a coloring $\chi_2: [\mathbb{Z}]^3 \to \{0,1\}$ by the rule $\chi_2(F) = 0$ if and only if $F$ is an arithmetic progression. Since $\chi_2$ is $\mathbb{Z}$-invariant and $\UU$ is $(\mathbb{Z},3)$-Ramsey, there is $Z\in \UU$, such that $Z\subset Y$ and $\chi_2|_{[Z]^3}\equiv i$. Clearly, $i=1$.

At last, $\chi_1|_{[Z]^4}\equiv 1$ and $\chi_2|_{[Z]^3}\equiv 1$ imply that $Z$ is 1-thin.
\end{proof}

A free ultrafilter $\UU$ on an Abelian group $G$ is said to be a $PS$-ultrafilter if, for any coloring $\chi: G\to \{0,1\}$ there exists $U\in \UU$ such that the set $PS(U)$ is $\chi$-monochrome, where $PS(U)=\{ a+b: a,b\in U, a \ne b \}$. Clearly, each selective ultrafilter on $G$ is a $PS$-ultrafilter. The following statements were proved in \cite{b6}, see also \cite[Chapter 10]{b2}. We denote by $PS(\UU)$ a filter with the base $\{PS(U): U\in \UU\}$. Then $\UU$ is a $PS$-ultrafilter if and only if $PS(\UU) = \UU + \UU$. If $G$ has no elements of order 2 then each $PS$-ultrafilter on $G$ is selective. A strongly summable ultrafilter on the countable Boolean group $B$ is a $PS$-ultrafilter but not selective. If there exists a $PS$-ultrafilter on some countable Abelian group then there is a $P$-point in $\omega^*$. It is easy to see that, an ultrafilter $\UU$ on a countable Boolean group $B$ is a $PS$-ultrafilter if and only if $\UU$ is $B$-Ramsey. Thus, a $B$-Ramsey ultrafilter needs not to be selective, but these ultrafilters cannot be constructed in ZFC with no additional assumptions.

Department of Cybernetics, Kyiv University, Volodimirska 64, Kyiv 01033, Ukraine

opetrenko72@gmail.com

I.V.Protasov@gmail.com

\end{document}